\documentclass[12pt]{amsart}
\usepackage[english]{babel}
\usepackage{amsthm}
\usepackage{inputenc}
\usepackage{amssymb}
\usepackage{color}
\usepackage{cite}

\newtheorem{theorem}{Theorem}[section]
\newtheorem{lemma}[theorem]{Lemma}

\begin{document}


\title[Local automorphisms on finite-dimensional
Lie algebras]{Local automorphisms on finite-dimensional Lie and
Leibniz  algebras}

\author[Ayupov]{Shavkat Ayupov}
\address{ V.I.Romanovskiy Institute of Mathematics, Uzbekistan Academy of Sciences, 81, Mirzo Ulughbek street, 100170  Tashkent,   Uzbekistan}
\email{sh$_{-}$ayupov@mail.ru}

\author[Kudaybergenov]{Karimbergen Kudaybergenov}
\address{Ch. Abdirov 1, Department of Mathematics, Karakalpak State University, Nukus 230113, Uzbekistan}
\email{karim2006@mail.ru}




\date{}
\maketitle

\begin{abstract}
We prove that a linear mapping on  the algebra \(\mathfrak{sl}_n\)
of all trace zero complex matrices  is a local automorphism if and
only if it is an automorphism or an anti-automorphism.   We also
show that a linear mapping on  a simple Leibniz algebra of the
form \(\mathfrak{sl}_n\dot +\mathcal{I}\)  is a local automorphism
if and only if it is an automorphism. We  give examples of
finite-dimensional nilpotent Lie algebras \(\mathcal{L}\) with
\(\dim \mathcal{L} \geq  3\) which admit local automorphisms
which are not automorphisms.

\\


\end{abstract}

\maketitle \thispagestyle{empty}


\section{Introduction}\label{sec:intro}

In last decades a series  of papers have been devoted to study of
mappings which are close to automorphism and derivation of
associative algebras (especially of operator algebras and
C*-algebras). Namely, the problems of describing  so-called
 local automorphisms (respectively, local derivations) and 2-local
automorphisms (respectively, 2-local derivations) have been
considered. Later similar problems were extended for non
associative algebras, in particular, for the case of Lie algebras.

Linear preserver problems (LPP) represent one of the most active
research areas in matrix theory. According to the linear character
of matrix theory, preserver problems mean here the
characterizations of all linear transformations on a given linear
space of matrices that leave certain functions, subsets,
relations, etc. invariant (see for example \cite{Molnar}). In this
paper we present some applications of LPP to the study of local
automorphisms of finite dimensional Lie and Leibniz algebras.

Let $\mathcal{A}$ be an associative  algebra. Recall that a linear
mapping $\Phi$ of $\mathcal{A}$ into itself is called a local
automorphism (respectively, a local derivation) if for every $x\in
\mathcal{A}$ there exists an automorphism (respectively, a
derivation) $\Phi_x$ of $\mathcal{A},$ depending on $x,$ such that
$\Phi_x(x)=\Phi(x).$ These notions were introduced and
investigated independently by Kadison~\cite{Kadison90} and Larson
and Sourour~\cite{Larson90}. Later, in 1997,
P.~\v{S}emrl~\cite{Semrl97}  introduced the concepts of 2-local
automorphisms and 2-local derivations. A map $\Phi:\mathcal{A}
\rightarrow \mathcal{A}$ (not linear in general) is called a
2-local automorphism (respectively, a 2-local derivation) if for
every $x, y\in \mathcal{A},$ there exists an automorphism
(respectively, a derivation) $\Phi_{x,y}:\mathcal{A} \rightarrow
\mathcal{A}$ (depending on $x, y$) such that
$\Phi_{x.y}(x)=\Phi(x),$ $\Phi_{x,y}(y)=\Phi(y).$   In
\cite{Semrl97}, P.~\v{S}emrl described 2-local derivations and
2-local automorphisms on the algebra $B(H)$ of all bounded linear
operators on the infinite-dimensional separable Hilbert space $H$
by proving that
 every 2-local automorphism (respectively, 2-local derivation) on $B(H)$ is an automorphism
 (respectively, a derivation).
 A similar result for finite-dimensional case appeared later
in~\cite{Kim04}. Further, in~\cite{AK}, a new techniques was
introduced to prove the same result for an arbitrary Hilbert space
$H$ (no separability is assumed.)

Afterwards the above considerations  gave  arise to similar
questions in von Neumann algebras framework. First positive
results have been obtained in ~\cite{AKNA} and ~\cite{AA2} for
finite and semi-finite von Neumann algebras respectively, by
showing that all 2-local derivations on these algebras are
derivations. Finally, in~~\cite{AK14}, the same result was
obtained for purely infinite von Neumann algebras. This completed
the solution of the above problem for arbitrary von Neumann
algebras.


It is natural to study the corresponding analogues  of these
problems for automorphisms or derivations of non-associative
algebras.

Let $\mathcal{L}$ be a Lie algebra. A derivation (respectively, an
automorphism) $\Phi$ of  $\mathcal{L}$ is a linear (respectively,
an invertible linear) map $\Phi:\mathcal{L} \rightarrow
\mathcal{L}$ which satisfies the condition  $\Phi([x, y])
=[\Phi(x), y] +[x, \Phi(y)]$ (respectively, $\Phi([x, y])
=[\Phi(x), \Phi(y)])$ for all $x, y\in \mathcal{L}.$ The set of
all automorphisms of a Lie algebra $\mathcal{L}$  is denoted by
$\mbox{Aut} \mathcal{L}.$

The notions of a local derivation (respectively, a local
automorphism) and a 2-local derivation (respectively, a 2-local
automorphism) for Lie algebras are defined as above, similar to
the associative case. Every derivation (respectively,
automorphism) of a Lie algebra $\mathcal{L}$ is a local derivation
(respectively, local automorphism) and a 2-local derivation
(respectively, 2-local automorphism). For a given Lie algebra
$\mathcal{L},$ the main problem concerning these notions is to
prove that they automatically become a  derivation (respectively,
an automorphism) or to give examples of local and 2-local
derivations or automorphisms of $\mathcal{L},$ which are not
derivations or automorphisms, respectively. For a
finite-dimensional semi-simple Lie algebra $\mathcal{L}$ over an
algebraically closed field of characteristic zero, the derivations
and automorphisms of $\mathcal{L}$ are completely described in
\cite{Humphreys}.

Recently in \cite{AK2016}
 we have proved  that every local derivation on semi-simple Lie algebras is
 a
derivation and gave examples of nilpotent finite-dimensional Lie
algebras with local derivations which are not derivations.

Earlier in~\cite{AKR15} the authors have proved that every 2-local
derivation on a semi-simple  Lie algebra $\mathcal{L}$ is a
derivation, and showed that each finite-dimension nilpotent Lie
algebra, with dimension larger than two,  admits a 2-local
derivation which is not a derivation.

In~\cite{Chen},  Chen and Wang initiated study of  2-local
automorphisms of  finite-dimensional Lie algebras. They prove that
if $\mathcal{L}$  is a simple Lie algebra of type
 $A_l\, (l\geq 1), D_l\, (l\geq4),$ or  $E_k\,
(k=6, 7, 8)$ over an algebraically closed field of characteristic
zero, then every 2-local automorphism of $\mathcal{L},$ is an
automorphism. Finally,
        in \cite{AK1601} Ayupov and Kudaybergenov generalized this result of \cite{Chen}
        and proved that every 2-local automorphism of a  finite-dimensional semi-simple
         Lie algebra over an algebraically closed field of characteristic zero is an
         automorphism. Moreover, they show also that every  nilpotent Lie algebra with
         finite dimension larger than two admits 2-local automorphisms which are not automorphisms.
It should be noted that similar problems for local automorphism of
finite-dimensional Lie algebras still remain open.

In the present paper we prove that a linear mapping on  the
algebra of all trace zero complex matrices \(\mathfrak{sl}_n\) is
a local automorphism if and only if it is an automorphism or an
anti-automorphism.   We also show that a linear mapping on  a
simple Leibniz algebra of the form \(\mathfrak{sl}_n\dot
+\mathcal{I}\) is a local automorphism if and only if it is an
automorphism. We also give examples of finite-dimensional
nilpotent Lie algebras \(\mathcal{L}\) with \(\dim L \geq  3\)
which admit local automorphisms  which are not automorphisms.

\section{Local automorphisms on \(\mathfrak{sl}_n\)}

In this section we study local automorphisms of the simple Lie
algebra \(\mathfrak{sl}_n\) of  all trace zero complex \(n\times
n\)-matrices.

Firstly we shall present the following two Theorems.

\begin{theorem}\label{localassoc} \cite[Theorem 2.2]{Larson90}.   \ If \(M_n(\mathbb{C})\)
is the algebra of all complex \(n\times n\)-matrices, then
\(\Delta : M_n(\mathbb{C})\to  M_n(\mathbb{C})\) is a local
automorphism iff  is an automorphism or an anti-automorphism,
i.e., either  \(\Delta(x) = axa^{-1}\)  or \(\Delta(x) =
ax^ta^{-1}\) for a fixed \(a\)  and for all \(x \in
M_n(\mathbb{C}).\)
\end{theorem}

Recall that here \(x^t\) denotes the transpose of the matrix
\(x.\)

We say that a matrix \(x\) is square-zero if \(x^2 = 0.\) A
mapping \(\Phi:\mathfrak{sl}_n\to \mathfrak{sl}_n\) preserves
square-zero matrices if \(x\in \mathfrak{sl}_n\) and \(x^2=0\)
imply \(\Phi(x)^2=0.\)

It is clear that the linear mappings defined as follows
\begin{eqnarray}\label{zero01}
\Phi(x)=\lambda axa^{-1},\, x\in \mathfrak{sl}_n
\end{eqnarray}
and
\begin{eqnarray}\label{zero02}
\Phi(x)=\lambda ax^t a^{-1},\, x\in \mathfrak{sl}_n.
\end{eqnarray}
preserve   square-zero matrices.

The following result about linear mappings preserving square-zero
due to \linebreak P.~\v{S}emrl~(see \cite{Semrl93}).

\begin{theorem} (see \cite[Corollary 2]{Semrl93}).\label{square-zero} Assume that
\(\Phi:\mathfrak{sl}_n \to \mathfrak{sl}_n\) is  a bijective
linear mapping preserving square-zero matrices. Then  \(\Phi\) is
either of the form~(1)  or of the form~(2).
\end{theorem}

It is well-known \cite{Jacob} that for any automorphism \(\Phi\)
on \(\mathfrak{sl}_n\) there exists an invertible matrix \(a\in
M_n(\mathbb{C})\) such that \(\Phi(x)=axa^{-1}\) or
\(\Phi(x)=-ax^t a^{-1}\) for all \(x\in \mathfrak{sl}_n.\)

\begin{theorem}\label{simplelie}
A linear mapping  \(\Delta: \mathfrak{sl}_n \to \mathfrak{sl}_n\)
is a local automorphism if and only if \(\Delta\) is either an
automorphism or an anti-automorphism, i.e., it has either the form
\begin{itemize}
\item[(i)] \(\Delta(x)=\pm axa^{-1},\, x\in \mathfrak{sl}_n\)

 or

\item[(ii)] \(\Delta(x)=\pm ax^t a^{-1},\, x\in \mathfrak{sl}_n.\)
\end{itemize}
\end{theorem}

\begin{proof} Let \(\Delta\) be a local automorphism on \(\mathfrak{sl}_n.\)
Let us show that \(\Delta\) is a bijective  linear map preserving
square-zero matrices.

Let \(x\in \mathfrak{sl}_n\) be a non zero element. By the
definition there exists an automorphism \(\Phi^x\) on
\(\mathfrak{sl}_n\) such that \(\Delta(x)=\Phi^x(x).\) Since
\(\Phi^x\)  is an automorphism, it follows that \(\Phi^x(x)\neq
0.\) Thus \(\Delta(x)\neq 0,\) and therefore the kernel of
\(\Delta\) is trivial. Hence \(\Delta\) is bijective.

Let now \(x\in \mathfrak{sl}_n\) be a square-zero matrix, that is
\(x^2=0.\)  By the definition of local automorphism there exists
an invertible matrix  \(a_x \in M_n(\mathbb{C})\) such that
\(\Delta(x)=a_x x a_x^{-1}\) or
 \(\Delta(x)= - a_x x^t a_x^{-1}.\)
We have
\begin{eqnarray*}
\Delta(x)^2 & = &  \left(a_x x a_x^{-1}\right)^2= a_x x^2
a_x^{-1}=0
\end{eqnarray*}
or
\begin{eqnarray*}
\Delta(x)^2 & = &  \left(-a_x x^t a_x^{-1}\right)^2 = a_x (x^2)^t
a_x^{-1}=0.
\end{eqnarray*}
So, \(\Delta\) is a bijective linear map preserving square-zero
matrices. By Theorem~\ref{square-zero}  there exist an invertible
matrix \(a\in M_n(\mathbb{C})\) and a nonzero scalar \(\lambda\)
such that
\begin{eqnarray}\label{cpone}
\Delta(x)=\lambda axa^{-1} \,\,\, or \,\,\, \Delta(x)=\lambda ax^t
a^{-1}
\end{eqnarray}
for all \(x\in \mathfrak{sl}_n.\)

Take the diagonal matrix  \(y=\left(%
\begin{array}{ccccc}
  1 & 0 & 0 & \cdots & 0 \\
  0 & -1  & 0 & \cdots  & 0 \\
  0 & 0 & 0 & \cdots  & 0 \\
  \cdot & \cdot & \cdot & \cdots  & \cdot \\
  0 & 0 & 0 & \cdots  & 0 \\
\end{array}%
\right)\) in \(\mathfrak{sl}_n.\)  By the definition of local
automorphism there exists an invertible matrix \(a_y \in
M_n(\mathbb{C})\) such that
\begin{eqnarray}\label{cptwo}
\Delta(y)=a_y y a_y^{-1} \,\,\, or \,\,\, \Delta(y)=-a_y y^t
a_y^{-1}.
\end{eqnarray}

Let \(p_x(t)=\det(t \mathbf{1} - x)\) be  the characteristic
polynomial of the matrix \(x,\) where \(\mathbf{1}\) is the unit
matrix in \(M_n(\mathbb{C}).\)

We shall find the  characteristic polynomial of the matrix
\(\Delta(y),\) using the equalities \eqref{cpone} and
\eqref{cptwo}. Taking into account that \(y=y^t\) and comparing
the equalities \eqref{cpone} and \eqref{cptwo}, we obtain that
\(\Delta(y)=\lambda aya^{-1}=\pm a_y y a_y^{-1}.\) Then
\begin{eqnarray*}
p_{\Delta(y)}(t)  & = & p_{\lambda aya^{-1}}(t) =  \det
(t\mathbf{1}-\lambda aya^{-1})
=\det (a(t\mathbf{1}-\lambda y)a^{-1}) =\\
& =& \det (t\mathbf{1}-\lambda y)=(t-\lambda)(t+\lambda)t^{n-2}
\end{eqnarray*}
and
\begin{eqnarray*}
p_{\Delta(y)}(t)  & = & p_{\pm a_y ya_y^{-1}}(t) =  \det
(t\mathbf{1}\mp a_yya_y^{-1})
=\det (a_y(t\mathbf{1}\mp y)a_y^{-1}) =\\
& =& \det (t\mathbf{1}\mp  y)=(t+1)(t-1)t^{n-2}
\end{eqnarray*}
Thus \((t-\lambda)(t+\lambda)t^{n-2}=(t-1)(t+1)t^{n-2},\) and
therefore \(\lambda=\pm 1.\) So, \(\Delta(x)=\pm axa^{-1}\) or
\(\Delta(x)=\pm ax^t a^{-1}\) for all \(x\in \mathfrak{sl}_n.\)

Now we shall show that every anti-automorphism on
\(\mathfrak{sl}_n\) is a local automorphism. Let us first show
that the mappings defined by
\begin{eqnarray}\label{trans}
\Delta(x)=x^t,\, x\in \mathfrak{sl}_n
\end{eqnarray}
and
\begin{eqnarray}\label{minus}
\Delta(x)=-x,\, x\in \mathfrak{sl}_n
\end{eqnarray}
are local automorphisms.

Let \(x\in \mathfrak{sl}_n\) be an arbitrary matrix. By
Theorem~\ref{localassoc}, the mapping on \(M_n(\mathbb{C})\)
defined as \(z\to z^t\) is a local associative automorphism.
Therefore there exists an invertible matrix \(a_x\in
M_n(\mathbb{C})\) such that \(x^t=a_x x a_x^{-1}.\) Hence
\begin{center}
 \(x^t=a_x x a_x^{-1}\) and  \(-x=- a_x^t x^t (a_x^t)^{-1}.\)
\end{center}
This means that the mappings defined as (5) and (6) are local
automorphisms on~\(\mathfrak{sl}_n.\) Finally, since every
anti-automorphism on \(\mathfrak{sl}_n\) is a superposition of an
anti-automorphism of the form (5) or (6) and an inner
automorphism, it follows that every anti-automorphism on
\(\mathfrak{sl}_n\) is a local automorphism. The proof is
complete.
\end{proof}

\section{Local automorphisms of algebras \(\mathfrak{sl}_n\dot+ \mathcal{I}\)}

In this section we study local automorphisms of  simple Leibniz
algebras of the form \(\mathfrak{sl}_n\dot+  \mathcal{I}.\)

An algebra $(\mathcal{L},[\cdot,\cdot])$ over a field $\mathbb{F}$
is called a Leibniz algebra if it satisfies the property
\begin{center}
\([x,[y,z]]=[[x,y],z] - [[x,z],y]\) for all \(x,y \in
\mathcal{L},\)
\end{center}
which is called Leibniz identity.

For a Leibniz algebra $\mathcal{L}$, a subspace generated by
squares of its elements  \linebreak $\mathcal{I}=span\left\{[x,x]:
x\in \mathcal{L}\right\}$ due to Leibniz identity becomes an
ideal, and the quotient
$\mathcal{G}_\mathcal{L}=\mathcal{L}/\mathcal{I}$ is a Lie algebra
called liezation of $\mathcal{L}.$ Moreover,  $[ \mathcal{L},
\mathcal{I}]=0.$ In general, $[\mathcal{I}, \mathcal{L}]\ne 0$.
Since we are interested in Leibniz algebras which are not Lie
algebras, we will always assume that $\mathcal{I}\ne0.$

A Leibniz algebra $\mathcal{L}$ is called  \textit{simple}  if its
liezation is a simple Lie algebra and the ideal $\mathcal{I}$ is a
simple ideal. Equivalently, $\mathcal{L}$ is simple iff
$\mathcal{I}$ is the only non-trivial ideal of~$\mathcal{L}.$

Let  $\mathcal{G}$ be a Lie algebra  and  $\mathcal{V}$ a  (right)
$\mathcal{G}$-module. Endow the vector space \linebreak
$\mathcal{L}=\mathcal{G} \oplus \mathcal{V}$ with the bracket
product as follows:
$$
[(g_1, v_1), (g_2, v_2)]:=([g_1, g_2], v_1.g_2),
$$
where  $v. g$ (sometimes denoted as $[v, g]$)  is the action of an
element $g$ of $\mathcal{G}$ on $v\in \mathcal{V}.$ Then
$\mathcal{L}$ is a Leibniz algebra, denoted as $\mathcal{G}\ltimes
V.$

The following Theorem is the main result of this section.

\begin{theorem}\label{leibnizsimple}
Let \(\mathfrak{sl}_n\dot +\mathcal{I}\) be a simple Leibniz
algebra with \(\mathcal{I}\neq\{0\}.\) Then a linear mapping
\(\Delta : \mathfrak{sl}_n \dot +\mathcal{I}\to \mathfrak{sl}_n
\dot +\mathcal{I}\) is a local automorphism if and only if
\(\Delta\) is an automorphism.
\end{theorem}

First we give necessary notations concerning the  algebra
\(\mathfrak{sl}_n.\) Let \(\{e_{ij}: 1\leq i, j \leq n\}\) be the
system of matrix units in \(M_n(\mathbb{C}).\) A subalgebra
\[
\mathfrak{h}=\left\{h=\sum\limits_{i=1}^n a_i e_{ii}:
\sum\limits_{i=1}^n a_i=0\right\}
\]
is a Cartan subalgebra of \(\mathfrak{sl}_n.\) For any \(i, j\in
\{1, \ldots, n\}\) we have
\[
[h, e_{ij}] = (a_i - a_j)e_{ij}.
\]
Let \(\mathfrak{h}^\ast\) be the space of all linear functionals
on \(\mathfrak{h}.\) Denote by \(\epsilon_i,\, 1\leq i \leq n,\)
the elements of \(\mathfrak{h}^\ast\) defined by
\[
\langle \epsilon_i, h\rangle=a_i.
\]
The root system \(R\)   consists  the elements of the form
\(\epsilon_i - \epsilon_j, i \neq j,\)  and \(\mathbb{C}e_{ij}\)
are the corresponding root subspaces, moreover,
\(R^+=\left\{\alpha=\epsilon_i-\epsilon_j: i<j\right\}\) is the
set of all positive roots. Denote
\[
\alpha_i=\epsilon_i-\epsilon_{i+1},\  i=1,...,n-1.
\]
From the relation
\[
\pm(\epsilon_i-\epsilon_j)=\pm\sum\limits_{k=i}^{j-1}\alpha_k,\,\,
i<j
\]
we see that the set \(\Pi\) formed by the elements \(\alpha_i, i =
1,... , n-1,\) is a base of \(R\) (see \cite{Humphreys}).

By \cite[Lemma~2.2]{Wang}, there exists an element $h_0\in
\mathfrak{h}$ such that $\alpha(h_0) \neq \beta(h_0)$ for every
$\alpha, \beta\in R, \alpha\neq \beta.$  Such elements $h_0$ are
called \textit{strongly regular} elements of $\mathfrak{sl}_n.$
Again by \cite[Lemma~2.2]{Wang}, every strongly regular element
\(h_0\) is a \textit{regular semi-simple element}, i.e.
$$
\{x\in \mathfrak{sl}_n: [h_0, x]=0\}=\mathfrak{h}.
$$

\begin{lemma}\label{lm0}
Let $\Phi$ be an automorphism on \(\mathfrak{sl}_n\) and let
\(h_0\) be a strongly regular element in \(\mathfrak{h}\) such
that \(\Phi(h_0)=-h_0.\) Then there exists an invertible diagonal
matrix \(a \in M_n(\mathbb{C})\) such that \(\Phi(x)=- ax^t
a^{-1}.\)
\end{lemma}

\begin{proof} Let \(\Theta\) be an automorphism on
\(\mathfrak{sl}_n\) defined by \(\Theta(x)=-x^t,\, x\in
\mathfrak{sl}_n.\) Since \((\Theta\circ \Phi)(h_0)=h_0,\) by
\cite[Lemma 2.2]{AK1601}, the automorphism \(\Theta\circ \Phi\)
leaves  every element of \(\mathfrak{h}\) fixed. Further by
\cite[P. 109, Lemma 2]{ZheWan}, there exists an invertible
diagonal matrix \(b \in M_n(\mathbb{C})\) such that \((\Theta\circ
\Phi)(x)= bx b^{-1}.\) Thus \(\Phi(x)= -b^tx^t (b^t)^{-1}\) for
all \(x\in \mathfrak{sl}_n.\)  The proof is complete.
\end{proof}

Let \(\mathcal{S}\) be a simple Lie algebra. For any
$\mathcal{S}$-module $\mathcal{I}$ and any automorphism $\sigma$
of $\mathcal{S},$ we define the new $\mathcal{S}$-module structure
$\mathcal{I}^{\sigma}$ on $\mathcal{I},$ given by the action
$$
v\cdot x=[v, x]'=v\sigma(x),\,\,\forall\ v\in \mathcal{I}, x\in
\mathcal{S}.
$$
We know from \cite{Humphreys}  that $\mathcal{I}\simeq
\mathcal{I}^{\sigma}$ if $\sigma$ is an inner automorphism of
$\mathcal{S}.$ If $\sigma$ is not an inner automorphism of
$\mathcal{S}$, we generally do not have  $\mathcal{I}\simeq
\mathcal{I}^{\sigma}$.

Let \(\mathcal{L}=\mathcal{S}\dot +\mathcal{I}\) be a simple
complex Leibniz algebra. In \cite{AKOZ} it was proved that an
automorphism $\sigma$ of $\mathcal{S}$ can be extended to an
automorphism  $\varphi$ of $\mathcal{L}$ if and only if
$\mathcal{I}\simeq \mathcal{I}^{\sigma}$ as $\mathcal{S}$-modules.
In \cite{AKOZ}  we also have present an example which shows the
existence of automorphism of \(\mathcal{S}\) which can not be
extended to the whole algebra \(\mathcal{L}.\)

Let \(\Phi:\mathcal{L}\to \mathcal{L}\) be an automorphism. Then
\(\Phi\) can be represented as (see \cite{AKO})
\begin{eqnarray}\label{autoform}
\Phi=\left(%
\begin{array}{cc}
  \Phi_\mathcal{S} & 0 \\
  \Phi_{\mathcal{S}, \mathcal{I}}  & \Phi_\mathcal{I} \\
\end{array}%
\right),
\end{eqnarray}
where \(\Phi_\mathcal{S}\) is an automorphism of \(\mathcal{S},\)
\(\Phi_{\mathcal{S}, \mathcal{I}} \) is a \(\mathcal{S}\)-module
homomorphism from \(\mathcal{S}\) to
\(\mathcal{I}^{\Phi_\mathcal{S}}\) and \(\Phi_\mathcal{I}\) is a
\(\mathcal{S}\)-module isomorphism from \(\mathcal{I}\) onto
\(\mathcal{I}^{\Phi_\mathcal{S}},\) i.e.,
\[
\left[\Phi_{\mathcal{S},\mathcal{I}}(x),
\Phi_\mathcal{S}(y)\right]=
 \Phi_{\mathcal{S},\mathcal{I}}\left([x, y]\right)
 \]
and
\[
\left[\Phi_\mathcal{I}(i), \Phi_\mathcal{S}(y)\right]=
 \Phi_\mathcal{I}\left([i, y]\right)
\]
for all \(x, y\in
 \mathcal{S}\) and  \(i\in
 \mathcal{I}.\) Note that \(\Phi_{\mathcal{S}, \mathcal{I}}=\omega \theta \circ \Phi_\mathcal{S},\) (\(\omega\in \mathbb{C}\))
 where
\(\theta\) is a \(\mathcal{S}\)-module isomorphism from
\(\mathcal{S}\) onto \(\mathcal{I}^{\Phi_\mathcal{S}}\) and for
\(\dim \mathcal{S}\neq \dim \mathcal{I},\) we
 have \(\Phi_{\mathcal{S},\mathcal{I}}=0.\)

Thus every local automorphism \(\Delta\) on
$\mathcal{L}=\mathcal{S}\dot+ \mathcal{I}$ is also represented as
\begin{eqnarray}\label{autoloc}
\Delta=\left(%
\begin{array}{cc}
  \Delta_\mathcal{S} & 0 \\
  \Delta_{\mathcal{S}, \mathcal{I}}  & \Delta_\mathcal{I} \\
\end{array}%
\right),
\end{eqnarray}
where \(\Delta_\mathcal{S}\) is a local automorphism of
\(\mathcal{S}.\)

Below we shall use the following properties of local
automorphisms.  Let \(\Delta\) be a local automorphism of
$\mathcal{S}\dot+ \mathcal{I}$  and let
\(x=x_\mathcal{S}+x_\mathcal{I}\in \mathcal{S} + \mathcal{I}.\)
Take an automorphism \(\Phi^{x}\) such that \(\Delta(x)=\Phi^x
(x).\) Then (7) and (8) imply that
\(\Phi^{x}_\mathcal{S}(x_\mathcal{S})=\Delta_\mathcal{S}(x_\mathcal{S})\)
and
\(\Phi^{x}_{\mathcal{S},\mathcal{I}}(x_\mathcal{S})+\Phi^{x}_\mathcal{I}(x_\mathcal{I})=
\Delta_{\mathcal{S},\mathcal{I}}(x_\mathcal{S})+\Delta_\mathcal{I}(x_\mathcal{I}).\)

Let \(\mathcal{S}\dot +\mathcal{I}\) be a simple complex Leibniz
algebra. From the representation theory of semisimple Lie algebras
\cite{Humphreys} we have that a Cartan subalgebra $\mathcal{H}$ of
the simple Lie algebra $\mathcal{S}$ acts diagonalizable on
$\mathcal{S}$-module $\mathcal{I}:$
\[
\mathcal{I}=\bigoplus_{\beta\in \Gamma} \mathcal{I}_\beta,
\]
where
\begin{eqnarray*}
\mathcal{I}_\beta & = & \{i\in \mathcal{I}:  [i, h]=\beta(h)i,
\,\, \forall \, h\in \mathcal{H}\},\\
\Gamma & = & \{\beta \in \mathcal{H}^\ast: \mathcal{I}_\beta\neq
\{0\}\}
\end{eqnarray*}
 and $\mathcal{H}^\ast$ is the space of all linear functionals on $\mathcal{H}.$ Elements of $\Gamma$ are called
  weights of~$\mathcal{I}.$

\begin{lemma}\label{spi}
Let $\Phi$ be an automorphism on
\(\mathfrak{sl}_n\dot{+}\mathcal{I}\) and let \(h_0\) be a
strongly regular element in \(\mathfrak{h}\) such that
\(\Phi_{\mathfrak{sl}_n}(h_0)=h_0.\) Then there exists
\(\lambda_\beta\in \mathbb{C}\)  such that
\(\Phi(y_\beta)=\lambda_\beta u_\beta,\) where \(\beta\) is a
highest weight of \(\mathcal{I}.\)
\end{lemma}

\begin{proof}  Since \(\Phi_{\mathfrak{sl}_n}(h_0)=h_0,\) by
\cite[Lemma 2.2]{AK1601} for every root \(\alpha\in R\) there
exists non zero \(c_\alpha\in \mathbb{C}\) such that
\(\Phi_{\mathfrak{sl}_n}(e_\alpha)=c_\alpha e_\alpha.\) Then
\[
[\Phi(y_\beta), e_\alpha] = c^{-1}_\alpha [\Phi(y_\beta),
\Phi(e_\alpha)]= c^{-1}_\alpha \Phi ([y_\beta,
e_\alpha])=c^{-1}_\alpha \Phi (0)=0
\]
for each positive root \(\alpha.\) Since the highest weight
subspace \(\mathcal{I}_\beta\) is one dimensional, it follows that
\(\Phi(y_\beta)=\lambda_\beta y_\beta.\) The proof is complete.
\end{proof}

From now on \(\mathfrak{sl}_n \dot +\mathcal{I}\) is a simple
Leibniz algebra with \(\mathcal{I}\neq\{0\}.\)

\begin{lemma}\label{lm1}
Let  \(\Delta : \mathfrak{sl}_n \dot +\mathcal{I}\to
\mathfrak{sl}_n \dot +\mathcal{I}\) be a linear mapping such that
\(\Delta_{\mathfrak{sl}_n}(x)=x^t\) for all \(x\in
\mathfrak{sl}_n.\) Then \(\Delta\) is not a local automorphism.
\end{lemma}

\begin{proof} Suppose that
\(\Delta\) is a local automorphism. Take an element
\(x=h_0+y_\beta\in \mathfrak{sl}_n \dot +\mathcal{I},\) where
\(h_0\in\mathfrak{h}\) is a strongly regular element and
\(y_\beta\) is a highest weight vector of \(\mathcal{I},\) i.e.,
\([y_\beta, e_\alpha]=0\) for every positive root \(\alpha.\) Take
an automorphism \(\Phi^{h_0}\) on \(\mathfrak{sl}_n\dot
+\mathcal{I}\) such that \(\Delta(h_0)=\Phi^{h_0}(h_0).\) Then
\begin{eqnarray*}
h_0+\Delta_{\mathfrak{sl}_n, \mathcal{I}}(h_0) & = &
\Delta_{\mathfrak{sl}_n}(h_0)+\Delta_{\mathfrak{sl}_n, \mathcal{I}}(h_)=\Delta(h_0)=\\
&=& \Phi^{h_0}(h_0)=
\Phi^{h_0}_{\mathfrak{sl}_n}(h_0)+\Phi^{h_0}_{\mathfrak{sl}_n,
\mathcal{I}}(h_0).
\end{eqnarray*}
Thus
\[
\Phi^{h_0}_{\mathfrak{sl}_n}(h_0)=h_0,
\]
and therefore
\[
\Phi^{h_0}_{\mathfrak{sl}_n, \mathcal{I}}(h_0)=\omega_1
\theta(\Phi^{h_0}_{\mathfrak{sl}_n}(h_0))=\omega_1\theta(h_0)=y^{(1)}_0\in
\mathcal{I}_0.
\]
Further
\begin{eqnarray*}
\Delta(h_0+y_\beta) & = &
\Delta(h_0)+\Delta(y_\beta)=\Phi^{h_0}(h_0)+\Delta(y_\beta)=\\
&=& \Phi_{\mathfrak{sl}_n}^{h_0}(h_0)+\Phi_{\mathfrak{sl}_n,
\mathcal{I}}^{h_0}(h_0)+\Delta(y_\beta)=
h_0+y^{(1)}_0+\Delta(y_\beta).
\end{eqnarray*}
Take an automorphism \(\Phi^x\) on \(\mathfrak{sl}_n\dot
+\mathcal{I}\) such that \(\Delta(x)=\Phi^x(x).\) Then
\begin{eqnarray*}
\Delta(h_0+y_\beta) & = &
\Phi^x(h_0+y_\beta)=\Phi^x_{\mathfrak{sl}_n}(h_0)+\Phi^x_{\mathfrak{sl}_n,
\mathcal{I}}(h_0)+\Phi^x_{\mathcal{I}}(y_\beta).
\end{eqnarray*}
Comparing the last two equalities, we obtain that
\(\Phi^x_{\mathfrak{sl}_n}(h_0)=h_0,\) and
\[
\Phi^{x}_{\mathfrak{sl}_n,
\mathcal{I}}(h_0)=\omega_2\theta(\Phi^{x}_{\mathfrak{sl}_n}(h_0))=\omega_2\theta(h_0)=y^{(2)}_0\in
\mathcal{I}_0.
\]
By Lemma~\ref{spi} it follows that
\(\Phi^x_{\mathcal{I}}(y_\beta)=\lambda_\beta y_\beta.\) Hence
\[
\Delta(y_\beta)=\lambda_\beta y_\beta+y_0,
\]
where \(y_0=y_0^{(2)}-y_0^{(1)}\in \mathcal{I}_0.\)

Now let us take an element \(z=e_\alpha+y_\beta,\) where \(\beta\)
is the highest weight of \(\mathcal{I}\) and \(\alpha\in R\) is  a
positive root.  Taking into account that \(\left[y_\beta,
e_\alpha\right]=0,\) we have
\begin{eqnarray*}
\left[\Delta(e_\alpha+y_\beta), \Delta(e_\alpha+y_\beta)\right]& =
& \left[\Phi^z(e_\alpha+y_\beta),
\Phi^z(e_\alpha+y_\beta)\right]=\Phi^z\left(\left[e_\alpha+y_\beta,
e_\alpha+y_\beta\right]\right)=\\
&=& \Phi^z([y_\beta, e_\alpha])=0.
\end{eqnarray*}
On the other hand,
\begin{eqnarray*}
[\Delta_{\mathcal{S}, \mathcal{I}}(e_\alpha), e_{-\alpha}] & = &
[\Delta_{\mathcal{S}, \mathcal{I}}(e_\alpha),
\Delta_{\mathcal{S}}(e_\alpha)]=\\
&=& [\Delta_{\mathcal{S}}(e_\alpha)+\Delta_{\mathcal{S},
\mathcal{I}}(e_\alpha),
\Delta_{\mathcal{S}}(e_\alpha)+\Delta_{\mathcal{S},
\mathcal{I}}(e_\alpha)]=\\
&=& [\Delta(e_\alpha),
\Delta(e_\alpha)]=[\Phi^{e_\alpha}(e_\alpha),
\Phi^{e_\alpha}(e_\alpha)]=\Phi^{e_\alpha}([e_\alpha, e_\alpha])=0
\end{eqnarray*}
and
\begin{eqnarray*}
\left[\Delta(e_\alpha+y_\beta), \Delta(e_\alpha+y_\beta)\right]& =
& \left[\Delta(e_\alpha)+\Delta(y_\beta),
\Delta(e_\alpha)+\Delta(y_\beta)\right]=\\
&=& [\Delta_\mathcal{S}(e_\alpha)+\Delta_{\mathcal{S},
\mathcal{I}}(e_\alpha)+\Delta(y_\beta),
\\ & & \Delta_\mathcal{S}(e_\alpha)+\Delta_{\mathcal{S},
\mathcal{I}}(e_\alpha)+\Delta(y_\beta)]=\\
&=& [e_{-\alpha}+\Delta_{\mathcal{S},
\mathcal{I}}(e_\alpha)+\lambda_\beta y_\beta+y_0,
\\ & & e_{-\alpha}+\Delta_{\mathcal{S},
\mathcal{I}}(e_\alpha)+\lambda_\beta y_\beta+y_0]=\\
&= & \left[\lambda_\beta y_\beta, e_{-\alpha}\right]+[y_0,
e_{-\alpha}]\in \mathcal{I}_{\beta-\alpha}+\mathcal{I}_{-\alpha},
\end{eqnarray*}
Thus \(\left[\lambda_\beta  y_\beta, e_{-\alpha}\right]=0\) for
all positive roots~\(\alpha,\) and therefore \(\left[\lambda_\beta
y_\beta, e_{\gamma}\right]=0\) for all roots~\(\gamma\in R.\) This
means that \(\left[\lambda_\beta y_\beta,
\mathfrak{sl}_n\right]=0.\) Since \(\mathcal{I}\) is  an
irreducible \(\mathfrak{sl}_n\)-module, it follows that
\(\lambda_\beta=0.\) So
\(\Phi^x_\mathcal{I}(y_\beta)=\lambda_\beta y_\beta=0,\) which
contradicts that \(\Phi^x_\mathcal{I}\) is invertible, because
\(\Phi^x_\mathcal{I}\) is an autmorphism. The proof is complete.
\end{proof}

\begin{lemma}\label{lm2}
Let \(\Delta : \mathfrak{sl}_n \dot +\mathcal{I}\to
\mathfrak{sl}_n \dot +\mathcal{I}\) be a linear mapping such that
\(\Delta_{\mathfrak{sl}_n}(x)=-x\) for all \(\mathfrak{sl}_n.\)
Then \(\Delta\) is not a local automorphism.
\end{lemma}

\begin{proof} Assume \(\Delta\) is a local automorphism.
Take an element \(x=h_0+y_\beta\in \mathfrak{sl}_n \dot
+\mathcal{I},\) where \(h_0\in\mathfrak{h}\) is a strongly regular
element and \(y_\beta\) is a highest weight vector of
\(\mathcal{I},\) i.e., \([y_\beta, e_\alpha]=0\) for all positive
root \(\alpha.\) We have
\begin{eqnarray*}
\Delta(h_0+y_\beta) & = &
\Delta(h_0)+\Delta(y_\beta)=\Phi^h(h_0)+\Delta(y_\beta)=\\
&=& -h_0+y^{(1)}_0+\Delta(y_\beta).
\end{eqnarray*}
Take an automorphism \(\Phi^x\) on \(\mathfrak{sl}_n\dot
+\mathcal{I}\) such that \(\Delta(x)=\Phi^x(x).\) Then
\begin{eqnarray*}
\Delta(h_0+y_\beta) & = &
\Phi^x(h_0+y_\beta)=\Phi^x_{\mathfrak{sl}_n}(h_0)+\Phi^x_{\mathfrak{sl}_n,
\mathcal{I}}(h_0)+\Phi^x_{\mathcal{I}}(y_\beta).
\end{eqnarray*}
Comparing the last two equalities, we obtain that
\[
\Phi^x_{\mathfrak{sl}_n}(h_0)=-h_0.
\]
Thus \(\Phi^x_{\mathfrak{sl}_n, \mathcal{I}}(h_0)=y^{(2)}_0\in
\mathcal{I}_0\) and \(\Phi^x_{\mathcal{I}}(y_\beta)=\lambda_\beta
y_{-\beta}.\) Hence
\[
\Delta(y_\beta)=\lambda_\beta y_{-\beta}+y_0.
\]

Let \(x=e_\alpha+y_\beta,\) where \(\beta\) be the highest weight
of \(\mathcal{I}\) and \(\alpha\) be a positive  root of
\(\mathfrak{sl}_n.\) As in the proof of the previous  Lemma, we
obtain that
\begin{eqnarray*}
\left[\Delta(e_\alpha+y_\beta), \Delta(e_\alpha+y_\beta)\right]& =
& \left[\Delta(e_\alpha)+\Delta(y_\beta),
\Delta(e_\alpha)+\Delta(y_\beta)\right]=\\
&=& [-e_{\alpha}+\Delta_{\mathfrak{sl}_n,
\mathcal{I}}(e_{\alpha})+\lambda_\beta y_{-\beta}+y_0,
\\
& & -e_{\alpha}+\Delta_{\mathfrak{sl}_n,
\mathcal{I}}(e_{\alpha})+\lambda_\beta y_{-\beta}+y_0]=\\
&=&-\lambda_\beta \left[y_{-\beta}, e_{\alpha}\right]+[y_0,
e_{-\alpha}]
\end{eqnarray*}
and
\begin{eqnarray*}
\left[\Delta(e_\alpha+y_\beta), \Delta(e_\alpha+y_\beta)\right]& =
& \left[\Phi^x(e_\alpha+y_\beta),
\Phi^x(e_\alpha+y_\beta)\right]=\Phi^x\left(\left[e_\alpha+y_\beta,
e_\alpha+y_\beta\right]\right)=\\
&=& \Phi^x([y_\beta, e_\alpha])=0,
\end{eqnarray*}
because \(\left[y_\beta, e_\alpha\right]=0.\) Thus \(\lambda_\beta
\left[y_{-\beta}, e_{\alpha}\right]=0\) for all \(\alpha.\) Hence
\(\lambda_\beta=0\) and therefore  \(\Delta(y_\beta)=0,\) which
contradicts the inversibility of  \(\Delta.\)  The proof is
complete.
\end{proof}

The following Lemma is a particular case of \cite[Theorem
6.9]{AKO}.

\begin{lemma}\label{lm3}
Let \(\Delta : \mathfrak{sl}_n \dot +\mathcal{I}\to
\mathfrak{sl}_n \dot +\mathcal{I}\) be a local automorphism such
that \(\Delta_{\mathfrak{sl}_n}(x)=x\) for all \(x\in
\mathfrak{sl}_n.\) Then \(\Delta\) is an automorphism.
\end{lemma}

Denote by \(\tau:\mathfrak{sl}_n \to \mathfrak{sl}_n\) the
automorphism of \(\mathfrak{sl}_n\) defined by
\[
\tau(x)=-x^t,\, x\in \mathfrak{sl}_n.
\]

\begin{lemma}\label{lm4}
Let \(\Delta : \mathfrak{sl}_n \dot +\mathcal{I}\to
\mathfrak{sl}_n \dot +\mathcal{I}\) be a local automorphism such
that \(\Delta_{\mathfrak{sl}_n}=\tau.\)  Then \(\mathcal{I}\cong
\mathcal{I}^{\tau}\) and \(\Delta\) is an automorphism.
\end{lemma}

\begin{proof}
Let \(\Delta : \mathfrak{sl}_n \dot +\mathcal{I}\to
\mathfrak{sl}_n \dot +\mathcal{I}\) be a local automorphism such
that \(\Delta_{\mathfrak{sl}_n}(x)=-x^t\) for all \(x\in
\mathfrak{sl}_n.\) Let us show that \(\mathcal{I}\cong
\mathcal{I}^{\tau}.\)

Take an automorphism \(\Phi^{h_0}\) of \(\mathfrak{sl}_n\dot
+\mathcal{I}\) such that
\(\Phi_{\mathfrak{sl}_n}^{h_0}(h_0)=\Delta_{\mathfrak{sl}_n}(h_0)=-h_0,\)
where \(h_0\in \mathfrak{h}\) is a strongly regular element. By
Lemma~\ref{lm0} it follows that
\(\Phi_{\mathfrak{sl}_n}^{h_0}(x)=-a x^t a^{-1}\) for all \(x\in
\mathfrak{sl}_n,\) where \(a\) is an invertible diagonal matrix.
Then \(\widetilde{\Phi}_a^{-1}\circ \Phi^{h_0}\) is an
automorphism of \(\mathfrak{sl}_n\dot +\mathcal{I}\) such that
\(\left(\widetilde{\Phi}_a^{-1}\circ
\Phi^{h_0}\right)_{\mathfrak{sl}_n}(x)=-x^t\) for all \(x\in
\mathfrak{sl}_n,\) where \(\widetilde{\Phi}_a\) is an extension
onto \(\mathfrak{sl}_n\dot +\mathcal{I}\) of the inner
automorphism of \(\mathfrak{sl}_n,\) generated by the element
\(a.\) Thus the restriction \(\widetilde{\Phi}_a^{-1}\circ
\Phi^{h_0}|_{\mathcal{I}}\)  is a \(\mathfrak{sl}_n\)-module
isomorphism from \(\mathcal{I}\) onto \(\mathcal{I}^{\tau},\) that
is  \(\mathcal{I}\cong \mathcal{I}^{\tau}.\) Further \(\Delta
\circ\widetilde{\Phi}_a^{-1}\circ \Phi^{h_0}\) is a local
automorphism such that \(\left(\Delta
\circ\widetilde{\Phi}_a^{-1}\circ
\Phi^{h_0}\right)_{\mathfrak{sl}_n}\) is an identical map on
\(\mathfrak{sl}_n.\) Lemma~\ref{lm3} implies that
 \(\Delta\) is an automorphism. The proof is complete.
\end{proof}

\textit{Proof of Theorem~\ref{leibnizsimple}.} It is suffices to
show that every  local automorphism on \(\mathfrak{sl}_n\dot
+\mathcal{I}\) is an automorphism.

Let \(\Delta\) be a local automorphism on \(\mathfrak{sl}_n\dot
+\mathcal{I}.\) By Theorem~\ref{simplelie}, the restriction
\(\Delta_{\mathfrak{sl}_n}\) is either an automorphism or an
anti-automorphism, i.e., there exists an invertible matrix \(a\)
such that \(\Delta_{\mathfrak{sl}_n}(x)=\pm axa^{-1}\) or
\(\Delta_{\mathfrak{sl}_n}(x)=\pm ax^ta^{-1}\) for all \(x\in
\mathfrak{sl}_n.\) Thus \(\left(\widetilde{\Phi}_a^{-1}\circ
\Delta\right)_{\mathfrak{sl}_n}(x)=\pm x\) or
\(\left(\widetilde{\Phi}_a^{-1}\circ
\Delta\right)_{\mathfrak{sl}_n}(x)=\pm x^t,\) where
\(\widetilde{\Phi}_a\) is an extension onto \(\mathfrak{sl}_n\dot
+\mathcal{I}\) of the inner automorphism of \(\mathfrak{sl}_n,\)
generated by the element \(a.\) By Lemmata~\ref{lm1} and
\ref{lm2}, it follows that \(\left(\widetilde{\Phi}_a^{-1}\circ
\Delta\right)_{\mathfrak{sl}_n}\) can not be an anti-automorphism.
So it is an automorphism, i.e.,
\(\left(\widetilde{\Phi}_a^{-1}\circ
\Delta\right)_{\mathfrak{sl}_n}(x)=x\) or
\(\left(\widetilde{\Phi}_a^{-1}\circ
\Delta\right)_{\mathfrak{sl}_n}(x)=-x^t.\) Further
Lemmata~\ref{lm3}~and~\ref{lm4} imply that
\(\widetilde{\Phi}_a^{-1}\circ \Delta\) is an automorphism and
therefore \(\Delta\) is also an automorphism.  The proof is
complete.

\section{Local automorphisms  of filiform Lie  algebras}

In this section we consider a special class of nilpotent Lie
algebras, so-called filiform Lie algebras, and show that they
admit local automorphisms which are not automorphisms.

A Lie algebra $\mathcal{L}$ is called \textit{nilpotent} if
$\mathcal{L}^k=\{0\}$ for some $k \in \mathbb{N},$  where
$\mathcal{L}^0=\mathcal{L},$ $\mathcal{L}^k=[\mathcal{L}^{k-1},
\mathcal{L}],\, k\geq1.$

In particular, a nilpotent Lie algebra $\mathcal{L}$ is called
\textit{filiform} if $\dim \mathcal{L}^k=n-k-1$ for $1\leq k \leq
n-1.$

\begin{theorem}\label{pure}
Let $\mathcal{L}$ be a finite-dimensional filiform   Lie algebra
with $\dim \mathcal{L}\geq 3.$ Then  $\mathcal{L}$ admits a local
automorphism which is not an automorphism.
\end{theorem}

It is known \cite{Vergne} that there exists a basis $\{e_1, e_2,
\cdots,  e_n\}$ of  $\mathcal{L}$ such that
\begin{equation}\label{basicfili}
[e_1, e_i]=e_{i+1}
\end{equation}
for all $i\in \overline{2, n-1}.$

Note that a filiform Lie algebra $\mathcal{L}$ besides (9)  may
have also other non-trivial commutators.

From (9) it follows that  $\{e_{k+2}, \cdots, e_n\}$ is a basis in
$\mathcal{L}^k$ for all $1\leq k \leq n-2$ and $ e_n\in
Z(\mathcal{L}).$  Since $[\mathcal{L}^1,
\mathcal{L}^{n-3}]\subseteq \mathcal{L}^{n-1}=\{0\},$ it follows
that
\begin{equation}\label{nminus}
[e_i, e_{n-1}]=0
\end{equation}
for all $i=3, \cdots, n.$

Define a linear operator $\Phi$ on $\mathcal{L}$ by
\begin{equation}\label{filiformder}
\Phi\left(x\right)=x+\alpha x_2 e_{n-1}+x_3 e_n,\,
x=\sum\limits_{k=1}^n x_ke_k\in \mathcal{L},
\end{equation}
where $\alpha \in \mathbb{C}.$

\begin{lemma}\label{filider}
A linear operator $\Phi$ on $\mathcal{L}$ defined by~(11) is an
automorphism if and only if $\alpha=1.$
\end{lemma}

\begin{proof}   Let \(x=\sum\limits_{k=1}^n x_ke_k,\, y=\sum\limits_{k=1}^n y_ke_k \in \mathcal{L}.\)
Then \([x, y]=(x_1y_2-x_2y_1)e_3+z,\) where \(z\in
\mathcal{L}^2.\) Thus
\begin{eqnarray*}
\Phi([x, y])  & = & [x, y]+(x_1y_2-x_2y_1)e_n.
\end{eqnarray*}
Taking into account (10), we obtain that
\begin{eqnarray*}
[\Phi(x), \Phi(y)] & = & [x+\alpha x_2 e_{n-1}+x_3 e_n. y+\alpha
y_2 e_{n-1}+y_3 e_n]=\\
&=& [x,y]+\alpha(x_1y_2-x_2y_1)e_n.
\end{eqnarray*}
Comparing the last two equalities we obtain that the map  $\Phi$
is an automorphism if and only if $\alpha=1.$ The proof is
complete.
\end{proof}

Consider  the linear operator $\Delta$ defined by (11) with
$\alpha=0.$

\begin{lemma}\label{locfili}
The  linear operator $\Delta$ is a local automorphism  which is
not an automorphism.
\end{lemma}

\begin{proof} By Lemma~\ref{filider}, $\Delta$ is not an
automorphism.

Let us show that $\Delta$ is a local automorphism. Let
$\Psi_\beta$ be a linear operator on $\mathcal{L}$ defined by
\[
\Psi_\beta\left(u\right)=u+\beta u_2 e_n,\, u=\sum\limits_{k=1}^n
u_ke_k\in \mathcal{L}.
\]
It is clear that \(\Psi_\beta\) is bijective. Since
$\Psi_\beta|_{\mathcal{[\mathcal{L},\mathcal{L}]}}\equiv
\textrm{id}_{\mathcal{[\mathcal{L},\mathcal{L}]}}$ and \(e_n \in
Z(\mathcal{L}),\)  it follows that
$$
\Psi_\beta([x,y])=[x,y]=[\Psi_\beta(x), \Psi_\beta(y)]
$$
for all $x, y\in \mathcal{L}.$ So, $\Psi_\beta$ is an
automorphism.

Finally, for any $x=\sum\limits_{k=1}^n x_ke_k$ let us show an
automorphism that coincides  with \(\Delta\) at the point \(x.\)
Denote by $\Phi$ the automorphism  defined by (11) with
$\alpha=1.$

Case 1. $x_2=0.$ Then
\begin{eqnarray*}
\Phi(x) & = &  x+ x_3 e_n  =  \Delta(x).
\end{eqnarray*}

Case 2. $x_2\neq 0.$ Set $\beta=\frac{\textstyle x_3}{\textstyle
x_2}.$ Then
\begin{eqnarray*}
\Psi_\beta(x)  & = &  x+\beta x_2 e_n=x+x_3 e_n=\Delta(x).
\end{eqnarray*}
The proof is complete.
\end{proof}

\section{Acknowledgements}
The motivation to study the problems considered in this paper came
out from discussions made with Professor E.Zelmanov during the
Second USA--Uzbekistan Conference held at the Urgench State
University
 on August, 2017, which the authors gratefully acknowledge.
The authors  are indebted to Professor Mauro Costantini from the
University of Padova for valuable comments to the initial version
of the present paper.

\end{document}